\documentclass{article}
\usepackage[utf8]{inputenc}
\usepackage{hyperref}
\usepackage{amsmath}
\usepackage{amssymb}
\usepackage{amsfonts}
\usepackage{amsthm}
\usepackage{amscd}

\title{Computing mixed Schatten norm of completely positive maps}
\author{\text{Mohammad ShahverdiKondori}$^1$ \hspace{0.5cm} \text{Siu On Chan}$^2$ \\ \\
$^1$ \text{Sharif university of technology} \hspace{0.3cm}  $^2$ \text{CUHK Hong Kong}}
\date{} 

\newcommand{\pqnorm}[1]{ \left \lVert #1 \right \rVert _{p \rightarrow q}}
\newcommand{\pnorm}[2]{ \left \lVert #1 \right \rVert _ {#2} }
\newcommand{\fa}{ \frac{\pnorm{Ax}{q}}{\pnorm{x}{p}} }
\newcommand{\fain}{{\pnorm{Ax}{q} \big / \pnorm{x}{p}}}
\newcommand{\sAx}[1]{J_{p^*}(A^TJ_q(A#1))}
\newcommand{\phicn}{\Phi : \C^{n\times n} \rightarrow \C^{m\times m}} 
\newcommand{\R}{\mathbb{R}}
\newcommand{\C}{\mathbb{C}}

\newcommand{\cpm}{completely positive map}
\newcommand{\spsq}[1]{S_p \rightarrow S_q }
\newcommand{\spqnorm}[1]{\left \lVert #1 \right \rVert _{S_p \rightarrow S_q}}
\newcommand{\spnorm}[2]{\left \lVert #1 \right \rVert _ {S_{#2}}}
\newcommand{\fphi}{ \frac{\spnorm{\Phi(A)}{q}}{\spnorm{A}{p}} }
\newcommand{\fphin}{\spnorm{\Phi(A)}{q} / \spnorm{A}{p} }
\newcommand{\sphi}[1]{J_{S_{p^*}}(\Phi^*J_{S_q}(\Phi(#1)))}
\newcommand{\hn}[1]{H^n_+}
\newcommand{\psd}[1]{positive semidefinite}
\newcommand{\grchain}[2]{#1_1 \geq #1_2 \geq \cdots \geq #1_#2}
\newcommand{\inner}[2]{\left < #1,#2 \right >}
\newcommand{\eigdec}[1]{Q \Lambda ^{#1} Q ^ \dagger}
\newcommand{\js}[2]{J_{S_#2}(#1)}
\newcommand{\fuin}[1]{fully indecomposable}

\newtheorem{thm}{Theorem}[section]
\newtheorem{defi}{Definition}[section]
\newtheorem{lemma}{Lemma}[section]
\newtheorem{corollary}{Corollary}[section]

\begin{document}

\maketitle

\begin{center}
    \textbf{Abstract}
\end{center}
Computing $p \rightarrow q$ norm for matrices is a classical problem in computational mathematics and power iteration is a well known method for computing $p \rightarrow q $ norm for a matrix with nonnegative entries. Here we define an equivalent iteration method for computing $ S_p \rightarrow S_q $ norm for completely positive maps where $S_p$ is the Schatten $p$ norm. We generalize almost all of definitions, properties, lemmas etc. in the matrix setting to completely positive maps and prove an important theorem in this setting. 

\section{Introduction} 
Given a matix $ A  \in \R ^{m \times n} $ we define $\pqnorm{A}$ as 

$$ \pqnorm{A}=\max \fa $$ 
where the $\pnorm{ . }{p}$ and $\pnorm{ . }{q}$ are $ \ell ^p $ norms. Computing $\pqnorm{A}$ is a classical problem in computational mathematics.The best known method for this is a nonlinear power method, introduced by Boyd in \cite{boyd 74} and then further analyzed and extended for instance in \cite{4 computing, 18 computing, 31 computing, 48 computing}. However the best known result about this problem is the following 
\begin{thm}[Theorems 3.2 and 3.3, \cite{26 of computing}]\label{thm 1}
Let $A \in  \R^{ m \times n} $ be a matrix with nonnegative entries
and suppose that $A^TA$ has at least one positive entry per row. If $1 < q \leq p < \infty$, then, every
positive critical point of ${f_A(x) = \fain}$ is a global maximizer. Moreover, if either $p > q$ or $A^TA$ is irreducible,
then $f_A$ has a unique positive critical point $x^+$ and the power sequence 
$$x_0 = \frac{x_0}{\hspace{0.1 in}\pnorm{x_0}{p}} , \quad x_{k+1} = \sAx{x_k}, \quad k=0,1,2,\cdots $$
converges to $x^+$ for every positive starting point.
\end{thm}

We can consider completely positive linear maps between complex matrix spaces (see \cite{cpm pro}) as a generalization for positive matrices so it's interesting to investigate this results for \cpm s.That means in this note we are trying to find $\spsq$$ norm of a \cpm \vspace{0.1cm} that defines as follows 
$$\spqnorm{\Phi} = \max \fphi$$ 
where the max function is on the set of all $ n \times n $ Hermitian matrices and $S_p$ and $S_q$ are Schatten norms.

We generalize the power method that is defined for positive matrices, to \cpm s and prove this method converges to the value of $\spsq$$ norm.


\section{$H^n_+$ cone and \cpm s preliminaries}
In this section we talk about positive semidefinite cone and \cpm s properties. Le $H^n$ be the set of all Hermitian $n \times n $ matrices and $H^n_+$ (resp. $H^n_{++}$) be the set of positive semidefinite (resp. positive definite) Hermitian $n \times n $ matrices. Also we use $A^\dagger $ for representing conjugate transpose of matrix $A$.  

\begin{defi}(\cpm)
Let $\phicn$ be a function on the set of complex matrices then $\Phi$ is a \cpm \vspace{0.1cm} if and only if there exist matrices $V_i \in \C^{m\times n}$ such that 
$$\Phi(A) = \sum_{i=1}^{k}{V_iAV_i^{\dagger}} , \quad k \leq nm$$
Also represent the transpose of $\Phi$ by $\Phi^*$ as 
$$\Phi^* : \C^{m \times m} \rightarrow \C^{n \times n}, \quad \Phi^*(A) = \sum_{i=1}^k {V_i^\dagger A V_i}$$
\textbf{note.} It's not the original definition but these are equivalent (see \cite{cpm pro})
\end{defi}
For a completely positive map we define $\spsq{}$ norm as 
$$\spqnorm{\Phi} = \max \fphi$$
where $S_p$ is the Schatten $p$-norm and defines as 

\begin{equation*}
    \begin{split}
        \spnorm{A}{p} & = tr(|A|^p)^{\frac{1}{p}}\\
        & = \left ( \sum_{i=1}^{n}{\sigma_i(A)}^p \right ) ^ \frac{1}{p}
    \end{split}
\end{equation*}
where $|A| = \sqrt{A^{\dagger}A}$ and $\sigma_i$-s are singular values of A.\\
\begin{defi}(proper cone)
A cone K in a vector space on field $\R$ is called proper cone if it satisfies these conditions: 
\begin{itemize}
    \item K is convex.
    \item K is closed.
    \item K is solid, which means it has nonempty interior.
    \item K is pointed, which means that it contains no line (or equally $ x \in K, -x \in K \Rightarrow x = 0) $ 
\end{itemize}
\end{defi}
By \cite{psd cone} we know that $H^n_+$ is a proper cone in space of complex matrices so we can define a partial ordering on matrices with respect to $H^n_+$ (for instance see \cite{boyd convex},chapter2). So we use the notation $A\succeq  0 $ for saying $A$ is positive semidefinite and $A \succ 0 $ means $A$ is positive definite also $A\succeq  B $ means $A-B $ is positive semidefinite or equally $A-B \in \hn{} $. Also for vectors $x$ and $y$ the notation $x\geq y $ means $x-y \in \R^n_+$ (respectively $x>y$ means $x-y \in \R^n_{++}$).
\begin{defi}
For a \psd{} matrix $A$ let $A = Q\Lambda Q^\dagger $ be the eigendecomposition of $A$ then for a positive real number $p$ we define $A^p$ as 
$$A^p = Q\Lambda^p Q^\dagger $$
where $\Lambda ^ p$ is the diagonal matrix that has $p$-th power of eigenvalues of $A$ as it's diagonal entries.
\end{defi}
\begin{lemma}\label{lem 2.1}
For $A,B \in \hn{}$ and \cpm \vspace{0.1cm} $\Phi $ we have the following properties :
\begin{enumerate}
    \item If $A\succeq B$ and consider $\grchain{\lambda}{n}$ and $\grchain{\gamma}{n}$ are $A'$s and $B'$s eigenvalues respectively then we have $\forall i : \lambda_i\geq \gamma_i $.  
    \item If $A\succeq B$ then $\Phi(A)\succeq \Phi(B)$ .
\end{enumerate}
\end{lemma}
\begin{proof}
For first part consider $\{v_1, v_2, \cdots, v_n \}$ and $\{u_1,u_2,\cdots,u_n\}$ are set of eigenvectors of $A,B$ respectively such that $ Av_i = \lambda_i v_i, Bu_i = \gamma_i u_i $ so these are orthonormal basis for space of complex vectors. Now for proving $\lambda_i \geq \gamma_i$ consider $ P = \{c_iv_i + c_{i+1}v_{i+1} + \cdots + c_nv_n : \forall j \hspace{0.07 in} c_j \in \C \}$ and $ Q = \{ d_1u_1 + d2_u2 + \cdots + d_iu_i : \forall j \hspace{0.07 in} d_j \in \C \} $ and $T = P \hspace{0.02 in} \cap \hspace{0.02 in} Q $ the intersection of spanned spaces by last $n-i+1$ eigenvectors of A and first i eigenvectors of B. Note that $dim(P) = n-i+1 $ and $dim(Q) = i $ so $dim(T) > 0 $ and it means there exist $x \neq 0, \pnorm{x}{} = 1 $ in T, for this x it's easy to see $x^\dagger A x \leq \lambda_i$ and $x^\dagger B x \geq \gamma_i $ but $A\succeq B $ implies $x^\dagger A x \geq x^\dagger B x $ so $\lambda_i \geq \gamma_i$.\\
For the second part note that $\Phi$ is linear.

\end{proof}

Hilbert projective metric is a metric that is defined on rays in a real Banach space and is introduced by Hilbert in \cite{hilbert metric} and first time it defined on \psd{} cone by C. Liverani and M. P. Wojtkowski in \cite{first hilbert psd}. Consider K is a proper cone, for $x,y \in K $ we use notation $x\sim y $ if there exist positive real numbers c and C such that $cy\preceq_K x \preceq_K Cy $ and that means $x-cy , Cy - x \in K$. It's easy to see $\sim$ is an equivalency relation and the equivalency classes are called parts of K. Also the function $M : K \times K \rightarrow \R$ is defined as $M(x/y) = \inf { \left \{ \lambda > 0 : x\preceq _K \lambda y \right \} }$ then the Hilbert projective metric $d:K \times K \rightarrow \R$ defines as follows 
\[
d(x,y) = 
\begin{cases}
\ln M(x/y)M(y/x) & x \sim y \\
0 & x=y=0 \\
\infty & otherwise 
\end{cases}
\]
It's easy to see that for $ K = \hn{} $ if $ A,B \succ 0 $ then
$$ d(A,B) = \ln\left ( \pnorm{{B^{-\frac12}AB^{-\frac12}}}{}\pnorm{{A^{-\frac12}BA^{-\frac12}}}{} \right ) $$
where $\pnorm{A}{} $ equals greatest eigenvalue of A.  The Hilbert projective metric is a metric on rays that means $ d(A,B) = d(\alpha A, \beta B) $ for every positive $ \alpha , \beta $ you can find more properties about Hilbert metric in \cite{hilbert metric}. 

The proof of our main theorem is based on the Banach contraction principle. Thus, for a map $ \phicn $ we consider the Birkhoff contraction ratio $ \kappa(\Phi) \in [0,\infty] $ of $ \Phi $, defined as the
smallest Lipschitz constant of $\Phi$ with respect to $d$ (see \cite{banach cont}):
$$ \kappa(\Phi) = \inf \left \{ C > 0 : d(\Phi (A), \Phi (B)) \leq C d(A,B), \quad \forall A,B \in \hn{} \text{ such that } A \sim B \right\} $$
where if there exist $ A,B \in \hn{}$ such that $A \sim B$ and $\Phi(A) \nsim \Phi(B) $ then $\kappa(\Phi) = \infty $ but we know that this case never happend for \cpm s because for linear maps we have $\kappa(\Phi) \leq 1 $. Moreover one can easily show that if $\Phi$ is a \cpm \vspace{0.1cm} then $A \sim B $ implies $\Phi(A) \sim \Phi(B) $. 

\begin{thm}(Birkhoff-Hopf,\cite{birkhoff})
Let $\phicn$ be a \cpm \vspace{0.1cm} then we have: 
$$ \kappa(\Phi) = \tanh \left( \frac{\Delta (\Phi)}{4} \right ) $$
where $\Delta$ is the diameter of $\Phi$ and defines as follows 
$$ \Delta (\Phi) = \sup \{ d(\Phi(A), \Phi(B)) : A,B \in \hn{} , A \sim B \} $$
and with the convention of $ \tanh ( \infty ) = 1 $.
\end{thm}
So the Bikhpff-Hopf theorem tells us the contraction ratio of a completely positive map is always less than or equal to 1 and it equals 1 if and only if $ \Delta(\Phi) = \infty $. 

\section{Nonlinear Perron Frobenius theorem for $\spqnorm{\Phi}$}
In this section we generalize the approach, includes structure, lemmas, theorems, etc, as \cite{main} chapter 4 for proving an important theorem (Theorem \ref{main thm}) in computing $\spsq{}$ norm for \cpm s.\\
For the Schatten $p$-norm, derivative $J_{S_p}$ can be represented as follows (see \cite{17 cpmuting}):
$$ J_{S_p}(A) = \{B : \inner{A}{B} = \spnorm{A}{p}, \spnorm{B}{p^*} = 1 \}$$
where $ \inner{A}{B} = tr(A^\dagger B) $ is the Frobenius inner product on matices. Also we know that because the Schatten norm is Fréchet differentiable so $ J_{S_p} $ is single valued also it's easy to see that for positive semidefinite matrices, $ J_{S_p} $ satisfies following equation (see \cite{schatten})
$$ J_{S_p} (A) = \frac{\hspace{-0.1 in}\eigdec{p-1}} {\hspace{0.1 in}\spnorm{\eigdec{p-1}}{p^*}} $$ 
where $S_{p^*}$ is the dual norm such that $\frac1{p} + \frac1{p^*} = 1$.

\begin{lemma}\label{lem1}
Given matrix A with $\spnorm{A}{p} = 1 $, ${f_\Phi (A) = \fphin \neq 0}$. A is a critical point of function $ f_\Phi $ if and only if it is a fixed point of $ S_\Phi (A) = \sphi{A} $.
\end{lemma}
\begin{proof}
At first assume A is a critical point of $f_\Phi$ then with differentiation we find 
$$ \Phi ^* J_{S_q}(\Phi (A)) = f_\Phi (A) J_{S_p}(A)  $$
then by applying the $ J_{S_{p^*}} $ function to above equation we have the following 
$$ \sphi{A} = \frac{\hspace{-0.05 in} A}{\hspace{0.1 in}\spnorm{A}{p}} $$
because $\js{\js{A}{p*}}{p} = A / \spnorm{A}{p}$ and $\js{\alpha A}{p} = \js{A}{p}$ so by the assumption $\spnorm{A}{p} = 1 $ we conclude A is a fixed point of $S_\Phi$. Now assume A is fixed point of $S_\Phi$ then we have 
$ \sphi{A} = A $ so there exist $\lambda > 0 $ such that $\lambda \Phi ^* \js{\Phi (A)}{q} = \js{A}{p}$ so by definition of $\js{A}{p}$ we have 
\begin{equation}\label{inner}
     \lambda ^ {-1} = \inner{A} {\Phi ^ * \js{\Phi (A)}{q}} = \inner{\Phi (A)}{\js{\Phi (A)}{q}} = \spnorm{\Phi (A)}{q} = f_\Phi (A)
\end{equation}
where the second equality holds because for linear function $\Phi$  we have $\inner{\Phi (A)}{B} = \inner{A}{\Phi ^* (B)}$ and the last equality holds because $\spnorm{A}{p} = 1$ so from \eqref{inner} we conclude A is a critical point of $ f_\Phi  $.
\end{proof}
\begin{defi}(see \cite{fully}, Proposition C.6) 
Let $\phicn$ be a \cpm \vspace{0.1cm} then we call it \fuin{} if for all singular, but nonzero $ A\succeq 0 $, $ rank(\Phi (A)) > rank(A) $ and $ \Phi $ is \fuin{} if and only if $ \Phi ^* $ is \fuin{}.
\end{defi}
 \begin{lemma}\label{lem2}
    Let $\phicn$ be a \cpm \vspace{0.1cm} and P be a part of $\psd{}$ cone such that for every $A \in P $ we have $ \Phi ^* (\Phi (A)) \in P $. If $ \kappa (\Phi) \leq \tau < 1 $ then $ \sphi{A} $ has a unique fixed point $ X $ in P and the following power method converges to $ X $ for any starting point $ A \in P $ :
    $$ A_0 = A , \quad A_{k+1} = S_\Phi (A_k), \quad k = 1,2,\cdots $$
 \end{lemma}
 \begin{proof}
 Note that if $ A \in P $ then $ \js{A}{p} \in P $ because P is a part of $\hn{}$ and $A \sim \js{A}{p}$ so the assumption $ \Phi ^* (\Phi (A)) \in P $ implies $ S_\Phi (A) \in P $ so by assumption $ \kappa (\Phi) \leq \tau < 1 $ we can use the Banach fixed point theorem for complete metric space $ (M,d) $ where $ M = P \hspace{0.02 in} \cap \hspace{0.02 in} \{A\succeq 0 : \spnorm{A}{p} = 1\} $ (for proof see \cite{complete metric}) and conclude $S_\Phi $ has a unique fixed point X in P and the power method converges to X for any starting point.
 \end{proof}
 
 Note that it's not enough to say the power method always converges to the maximizer of $f_\Phi$ because now it's possible for $f_\Phi$ to have more than one critical points because we only proved in every part of $\hn{}$ it has at most one critical point and for our claim we need some more assumptions that you can see in following lemmas. 
 
 \begin{lemma}\label{lem3}
 For any \cpm \vspace{0.1cm} $\Phi$ the global maximum of $ f_\Phi $ is attained in $\hn{}$.
 \end{lemma}
 \begin{proof}
 We know that $ \spnorm{|A|}{p} = \spnorm{A}{p} $ and also $ \spnorm{\Phi (A)}{p} \leq \spnorm{\Phi (|A|)}{p} $ for any $ A \in H^{n} $ because for every Hermitian matrix A, eigenvalues of $|A|$ are absolute values of eigenvalues of A so $A \preceq |A|$ and by Lemma \ref{lem 2.1} we have $\Phi(A) \preceq \Phi(|A|)$ and then because of monotonicity of Schatten norms we have $ \spnorm{\Phi (A)}{p} \leq \spnorm{\Phi (|A|)}{p} $. So we have the following 
 $$ f_\Phi(A) = \fphi = \frac{\spnorm{|\Phi(A)|}{q}}{\spnorm{|A|}{p}} \leq \frac{\spnorm{\Phi(|A|}{q}}{\spnorm{|A|}{p}} = f_\Phi(|A|) $$
 so if $A$ is a maximizer of $f_\Phi$ then $ f_\Phi (A) \leq f_\Phi (|A|) $ which concludes the proof. 
 \end{proof}

\begin{lemma}\label{lem4}
For every $ A \in \hn{} , p \in \R^N_+ $ we have $ S_\Phi(A) \sim \Phi ^* \Phi (A).$
\end{lemma}
\begin{proof}
It's easy to see $ A \sim \js{A}{p} $ and if $ A \sim B $ then $ \Phi(A) \sim \Phi(B) $ so we have 
$$ \Phi (A) \sim \js{\Phi (A)}{q} \Rightarrow \Phi ^* \Phi (A) \sim \Phi ^* \js{\Phi (A)}{q} \Rightarrow \Phi ^* \Phi (A) \sim S_\Phi(A)$$
\end{proof}

\begin{lemma}\label{lem5}
Let $ \phicn $ be a \cpm \vspace{0.1cm} and suppose that $ \Phi ^* \Phi $ is \fuin{} then $ S_\Phi (A) \in H^n_{++}$ for every $A \in H^n_{++}$ and every positive semidefinite critical point A of $f_\Phi$ is positive definite. 
\end{lemma}
\begin{proof}
For proving $ S_\Phi (H^n_{++}) \subseteq H^n_{++} $ consider there exist a matrix $ A \in (H^n_{++}) $ such that $ S_\Phi (A) \nsubseteq (H^n_{++}) $ then $ rank(S_\Phi (A)) < rank(A) $ but from Lemma \ref{lem4} it's known $ S_\Phi(A) \sim \Phi ^* \Phi (A)$ so $ rank(\Phi ^* \Phi (A)) < n $ and it's impossible because of the assumption $ \Phi ^* \Phi $ is \fuin{}. Now consider $ S_\Phi (A) = A $ and $ A \in \hn{} $ so $ \Phi ^* \Phi (A) \sim A $ and this implies $rank(A) = n $ so A is positive definite. 
\end{proof}

\begin{thm}\label{main thm}
Let $\phicn$ be a \cpm \vspace{0.1cm} and $\Phi ^* \Phi $ \fuin{}. If $ \kappa (S_\Phi) \leq \tau < 1 $ then $ f_\Phi $ has a unique critical point $X \in H^n_{+} $ and $f_\Phi (X) = \spqnorm{\Phi}$ and X is positive definite. Moreover the following power method converges to X for any starting point $A \in \hn{}$
$$ A_0 = A , \quad A_{k+1} = S_\Phi (A_k), \quad k = 1,2,\cdots $$
\end{thm}
\begin{proof}
From Lemma \ref{lem3} we know that $f_\Phi$ has a miximizer $X \in \hn{}$ and from Lemma \ref{lem5} we know X is positive definite also from Lemma \ref{lem2} we know that the power method converges to fixed point of $S_\Phi$ and from Lemma \ref{lem1} it's the unique maximizer of $f_\Phi$ in $\hn{}$ so we are done. 
\end{proof}
\begin{corollary}
By definition of $\kappa $ it's clear $ \kappa(S_\Phi) \leq \kappa(J_{S_{p^*}}) \kappa(\Phi ^*) \kappa(J_{S_q}) \kappa(\Phi) $ on the other hand we have $ \kappa(J_{S_p}) = p-1 $ so $ \kappa(S_\Phi) \leq \kappa(\Phi ^*) \kappa(\Phi) \frac{q-1}{p-1}$ so from Theorem \ref{main thm} the convergence to $\spsq{}$ is proven for case $p > q$.
\end{corollary}
So this result is close to result for classical setting of computing $p \rightarrow q $ norm for nonnegative matrices (Theorem \ref{thm 1}) and we can see the similarity between irreduciblity in matrices and fully indecomposablity in \cpm s but it's remained to generalize the result for $p = q $ case, this case isn't proven yet because it's possible to have $ \kappa(\Phi ^*) = \kappa(\Phi) = 1 $ then $S_\Phi$ isn't a contraction and Banach fixed point theorem isn't useful.\\
So it's interesting to find the biggest set of \cpm s $\Phi$ with $\kappa (\Phi) < 1$.
However we finish this note by a generalization for the case $ p = q$ for positively improving maps.

\begin{defi}
Let $\phicn$ be a \cpm \vspace{0.1cm} then it is positively improving if for every $A \in \hn{}$ we have $ \Phi(A) \in H^n_{++} $. 
\end{defi}
\begin{thm}
If $\phicn$ be a positively improving \cpm \vspace{0.1cm} then $\kappa(\Phi) < 1 $. 
\end{thm}
\begin{proof}
From Birkhoff-Hopf theorem it's enough to prove $\Delta (\Phi) < \infty $ and also it's enough to prove $\pnorm{{\Phi(B)^{-\frac12}\Phi(A)\Phi(B)^{-\frac12}}}{} < \infty $ for every $A,B \in \hn{}$ with $trace(A) = trace(B) = 1 $ because Hilbert metric is defined on rays and also by definition $\Phi(B), \Phi(A)$ are positive definite so $\Phi(B)^{-\frac12}$ exist. It's known for every $A$ and $B \in H^n_+$ we have $\pnorm{AB}{} \leq \pnorm{A}{}\pnorm{B}{}$ so $\pnorm{{\Phi(B)^{-\frac12}\Phi(A)\Phi(B)^{-\frac12}}}{} \leq {\lambda_{max}(\Phi(A))} / {\lambda_{min}(\Phi(B))} $ where $\lambda_{max}$ and $\lambda_{min}$ are the largest and the smallest eigenvalues. So it's enough to show there exist numbers $c,C > 0$ such that $ \lambda_{max}(\Phi(A)) < C $ and $\lambda_{min}(\Phi(A)) > c $ for any $A \in \hn{} $ with $trace(A) = 1 $, for proving this note that $\Phi$ is a continuous function on the compact set $\{A : A \in \hn{} , trace(A) = 1 \}$ so the image of $\Phi $ is also compact so $ \lambda_{max}(\Phi(A)) < C $ and also we know the image is a subset of inside of positive semidefinite cone that contains it's boundary so the image of $\Phi$ has a positive distance $c > 0$ with the boundary and it means for every matrix $A$ the matrix $\Phi (A) - cI $ is positive definite and it implies $\lambda_{min} (\Phi (A)) > c $.
\end{proof}

\end{document}